\documentclass[11pt]{article}

\usepackage[a4paper,margin=1in]{geometry}
\usepackage{amsmath,amssymb,amsthm}
\usepackage{booktabs}
\usepackage{hyperref}
\usepackage{authblk}
\usepackage{float}

\newcommand{\ZZ}{\mathbb Z}
\newcommand\FF{\mathbb F}

\newcommand\RR{\mathbb R}
\newcommand\QQ{\mathbb Q}

\newcommand{\F}{\mathbf F}
\newcommand{\G}{\mathbf G}
\newcommand{\ii}{i}
\DeclareMathOperator{\den}{den}

\theoremstyle{plain}
\newtheorem{theorem}{Theorem}
\newtheorem{lemma}{Lemma}
\newtheorem{proposition}{Proposition}
\newtheorem{corollary}{Corollary}

\theoremstyle{definition}
\newtheorem{definition}{Definition}

\theoremstyle{remark}
\newtheorem{remark}{Remark}

\title{A $p$-adic ($p\equiv 3 \pmod 4$) depth-$5$ supercongruence for Gaussian $p$-th power sums over a square }
\author[1]{Nikita Kalinin}
\author[2]{Faith Shadow Zottor}
\affil[1]{Guangdong Technion-Israel Institute of Technology (GTIIT),
$241$ Daxue Road, Shantou, Guangdong Province $515603$, P.R. China, Technion-Israel Institute of Technology, Haifa, 32000, Haifa District, Israel, nikaanspb@gmail.com}
\affil[2]{Guangdong Technion-Israel Institute of Technology (GTIIT),
$241$ Daxue Road, Shantou, Guangdong Province $515603$, P.R. China, faith.zottor@gtiit.edu.cn}

\date{}

\begin{document}
\maketitle

\begin{abstract}
For an odd prime $p$, define
\[
\G_n(p)=
\sum_{a=1}^{p-1}\sum_{b=1}^{p-1}(a+b\ii)^n
\in\ZZ[\ii].
\]
We study the $p$-adic valuation of these Gaussian power sums for $n\leq p, n=r(p-1)$ and show
that it is governed by the interaction of the fourfold symmetry of the
square, ordinary power-sum congruences, and Bernoulli numbers.



If $p\equiv3\pmod4$ and $p\ge7$, then we prove an unexpectedly deep supercongruence
\[
\G_p(p)\equiv
-\frac{p^5}{12}(p-1)^2(p-2)(1-\ii)B_{p-3}
\pmod{p^6}.
\]
\end{abstract}

\bigskip\textbf{Keywords:} Wolstenholme prime, Power sum, Bernoulli number, Supercongruences.

\textbf{AMS Subject Classification(2020):}  11B99, 11A99, 11A07

\section{Introduction}

Formulas for finite power sums
\[
S_r(N):=\sum_{t=1}^{N} t^r
\]
were already systematically explored by J.~Faulhaber in the early seventeenth century and were later placed into a uniform framework by J.~Bernoulli \cite{bernoulli1713ars,knuth1993johann} via the introduction of the Bernoulli numbers and Bernoulli polynomials, yielding the identity
\[
S_r(N)=\frac{B_{r+1}(N+1)-B_{r+1}}{r+1}.
\]

Ever since Bernoulli, power sums have been a rich source of congruences and $p$-adic phenomena.  The archetypal example is the von Staudt--Clausen description of the denominators of Bernoulli numbers, which explains when $p$ can appear in denominators of Bernoulli numbers and underlies many ``coincidences'' in $p$-adic arithmetic, \cite{staudt1840beweis}. A second foundational theme is that Bernoulli numbers encode congruences in families (Kummer-type congruences) and, more broadly, measure irregularity phenomena for primes \cite{kummer1851allgemeine}.

In this paper, motivated by conjectures in \cite{kalinin2025wolstenholme}, we study a two-parameter \emph{Gaussian} analogue,
\[
\F_n(k,m)=\sum_{a=1}^{k}\sum_{b=1}^{m}(a+b\ii)^n\in\ZZ[\ii],
\qquad n\in\ZZ_{\ge 0},\ k,m\in\ZZ_{\ge 1},
\]
and focus in particular on the square specialization at an odd prime
\[
\G_n(p):=\F_n(p-1,p-1)=\sum_{a=1}^{p-1}\sum_{b=1}^{p-1}(a+b\ii)^n\in\ZZ[\ii],
\]
where \(p\) is an odd prime.


For a finite commutative unital ring $R$ and an integer $k\ge 1$, set
\[
S_k(R):=\sum_{x\in R} x^k\in R.
\]
Determining $S_k(R)$ is a natural extension of the classical power-sum problem over finite fields and residue rings, and it has been studied in generality for finite commutative unital rings; see, for instance, \cite{ayuso2015staudt} and the refinements in \cite{grau2017power}.  In the Gaussian setting, one considers
\[
\G_n(p)=\sum_{a,b=1}^{p-1}(a+bi)^n\in\ZZ[i].
\]
Ayuso--Grau--Oller-Marc\'en \cite{ayuso2015staudt} obtained a von Staudt-type description of $\mathbf{G}_n(k)\bmod k$, for $k$ not necessarily prime   (see also \cite{khare2019carlitz} for Carlitz Staudt-type results in finite rings).

 In particular, they develop a Gaussian analogue of the Carlitz--von Staudt theory: for general moduli $n$ they obtain an explicit congruence for the sum of $k$-th powers in the finite ring \(\ZZ_n[i]\), describing its residue class modulo $n$ by an explicit expression determined by the prime divisors of $n$ satisfying certain congruence conditions (in particular, by primes in specific congruence classes modulo $4$).  As a consequence, they give an explicit criterion for when the \emph{diagonal values} satisfy \(n\mid \G_k(n)\), and they prove that the set of such $n$ has an asymptotic density which they compute numerically to six digits.

We have a different emphasis: we fix a prime $p$ and study the \emph{$p$-adic depth} of $\mathbf{G}_k(p)$ inside $\ZZ[i]$.
Over a ring of characteristic $p$ it is often natural (and frequently forced by symmetry) that $S_k(R)$ vanishes in $R$, i.e.\ the corresponding lift is divisible by $p$.
Our goal is to quantify and explain \emph{higher} divisibility:
we ask for the largest $m$ such that
\[
\mathbf{G}_k(p)\equiv 0 \pmod{p^m}
\qquad\text{in }\ZZ[i],
\]
in other words, we seek \emph{supercongruences} for Gaussian power sums.
In the case $k=p$, we show that the depth can be unexpectedly large and is governed by the splitting behavior of $p$ in $\ZZ[i]$ together with a Bernoulli obstruction (at index $p-3$) in the inert case.

Our main theorem is as follows.

\begin{theorem}\label{thm:Gp-p3}
Let $p\ge7$ be a prime satisfying $p\equiv3\pmod4$. Then, in
$\ZZ_{(p)}[\ii]$,
\begin{equation}\label{eq:main-inert}
\G_p(p)\equiv
-\frac{p^5}{12}(p-1)^2(p-2)(1-\ii)B_{p-3}
\pmod{p^6}.
\end{equation}
Moreover,
\[
(1-\ii)p^5\mid\G_p(p)
\qquad\text{in }\ZZ[\ii].
\]

Since $p-1\nmid p-3$, the von Staudt--Clausen theorem implies
\[
p\nmid\den(B_{p-3}),
\]
so the reduction of $B_{p-3}$ modulo $p$ is well defined. Consequently,
\[
v_p(\G_p(p))\ge6
\quad\Longleftrightarrow\quad
B_{p-3}\equiv0\pmod p,
\]
and equivalently,
\[
v_p(\G_p(p))=5
\quad\Longleftrightarrow\quad
B_{p-3}\not\equiv0\pmod p.
\]
\end{theorem}

\medskip

\subsection*{Conceptual structure of the results}

Three mechanisms govern the congruences in this paper.

First, the square is invariant under interchange of its coordinates.
This gives
\[
\G_n(p)=\ii^n\overline{\G_n(p)},
\]
and hence
\[
\G_n(p)\in
\begin{cases}
\RR,&n\equiv0\pmod4,\\
(1+\ii)\RR,&n\equiv1\pmod4,\\
\ii\RR,&n\equiv2\pmod4,\\
(1-\ii)\RR,&n\equiv3\pmod4.
\end{cases}
\]
Thus the recurrent decomposition according to $n\bmod4$ is forced by
the fourfold symmetry of the summation domain.

Second, the binomial factorization (Lemma~\ref{lem:binomfactor})
\[
\G_n(p)
=
\sum_{j=0}^{n}
\binom nj\ii^j
S_{n-j}(p-1)S_j(p-1)
\]
reduces the two-dimensional Gaussian sum to ordinary power sums.
Since (see Lemma~\ref{lem:powersum_modp})
\[
S_m(p-1)\equiv0\pmod p
\qquad\text{unless }p-1\mid m,
\] 
higher divisibility is obtained from the simultaneous occurrence of
several such factors and from cancellations between the surviving
terms.

Third, Faulhaber's formula gives the $p$-adic expansion
\[
S_m(p-1)
=
\sum_{\ell=1}^{m+1}
\frac{p^\ell}{\ell}\binom{m}{\ell-1}B_{m-\ell+1}.
\]
The vanishing of odd Bernoulli numbers, together with the fourfold
symmetry, implies that for
\[
n=4t+r,\qquad r\in\{0,1,2,3\},
\]
the first symmetry-compatible term has order $p^{r+1}$ and is a
nonzero rational multiple of $B_{4t}$. Thus Bernoulli irregular pairs
are precisely the places where the generic mod-$4$ valuation pattern
can increase.

For multiples of $p-1$ at inert primes, an additional finite-field
symmetry appears. The map
\[
t\longmapsto\frac{t-\ii}{t+\ii}
\]
identifies $\mathbb P^1(\FF_p)$ with the norm-one subgroup of
$\FF_{p^2}^{\times}$. This converts the reduction of
$\G_{r(p-1)}(p)$ into a power sum on a cyclic group of order $p+1$.

Finally, at the special exponent $n=p$, the non-$p$-integral Bernoulli
number $B_{p-1}$ appears. For split primes its contribution survives,
giving depth $2$. For inert primes the $B_{p-1}$ terms cancel, and a
second involutive cancellation eliminates the bulk of the binomial
sum. The first remaining obstruction is therefore $B_{p-3}$ at depth
$5$.

\subsection*{Organization of the paper}

In Section~\ref{sec_2} we introduce the two--parameter Gaussian power sums
\[
\F_n(k,m)=\sum_{a=1}^{k}\sum_{b=1}^{m}(a+b\ii)^n\in\ZZ[\ii],
\qquad
S_r(N)=\sum_{t=1}^{N}t^r\in\ZZ,
\]
and establish a structural reduction of $\F_n(k,m)$ to the classical power sums via the binomial--power--sum factorization
(Lemma~\ref{lem:binomfactor}). This immediately implies that, for fixed $n$, the function $(k,m)\mapsto \F_n(k,m)$ is polynomial
in $(k,m)$ with coefficients in $\QQ(\ii)$ (Proposition~\ref{prop:polynomiality}). 

In Section~\ref{sec_3} we specialize to the prime--square sums
\[
\G_n(p)=\F_n(p-1,p-1)\in\ZZ[\ii],
\qquad
v_p(x+y\ii)=\min\{v_p(x),v_p(y)\},
\]
and prove our first arithmetic results: congruences modulo $p$ and $p^2$.
After recording two standard tools (power sums modulo $p$ and a Lucas--type vanishing for certain stretched binomial coefficients),
we prove a clean mod-$p$ dichotomy (Theorem~\ref{thm:modp}): if $(p-1)\nmid n$ then $\G_n(p)\equiv 0\pmod p$, whereas for
$n=r(p-1)$ with $1\le r<p$ one has $\G_n(p)\equiv 1+\ii^{\,r(p-1)}\pmod p$.
We then pass to reduction modulo $p^2$ in the range $1\le n\le p-2$, showing that only the endpoint terms in the binomial factorization
contribute (Theorem~\ref{thm:modp2_endpoint}). For even exponents $2\le n\le p-3$ this yields an explicit Bernoulli-number refinement
(Corollary~\ref{cor:evenBn}), which already forces additional $p$-divisibility in certain residue classes $n\bmod 4$.

In Section~\ref{sec_3} we also determine the small-exponent valuation
pattern.  For $n=4t+r\le p-2$, with $r\in\{0,1,2,3\}$, we obtain a
leading congruence of order $p^{r+1}$ whose coefficient is an explicit
$p$-adic unit times $B_{4t}$. Thus the generic mod-$4$ valuation
pattern is governed precisely by the Bernoulli irregular pairs
$(p,4t)$.  In the same section we prove that, for inert primes
$p\equiv3\pmod4$ and arbitrary $m\ge1$,
\[
v_p\bigl(\G_{m(p-1)}(p)\bigr)=0
\quad\text{if $m$ is even},
\]
whereas
\[
v_p\bigl(\G_{m(p-1)}(p)\bigr)\ge3
\quad\text{if $m$ is odd}.
\]

The next section develops a higher-precision $p$-adic analysis at the
special exponent $n=p$. Starting from an exact Faulhaber identity for
$S_r(p-1)$, we derive uniform truncations modulo $p^6$. We then split
the binomial decomposition of $\G_p(p)$ into endpoint, boundary, and
bulk contributions. In the split case $p\equiv1\pmod4$ this gives
\[
\G_p(p)\equiv p^2(1+\ii)\pmod{p^3},
\]
while in the inert case $p\equiv3\pmod4$ it gives the congruence
\eqref{eq:main-inert}. The valuation in the inert case is therefore
controlled by the residue of $B_{p-3}$ modulo $p$.

Finally, we prove closed formulas for $v_3(\G_n(3))$ and
$v_5(\G_n(5))$ for all $n\ge1$, and we record a universal
$p$-adic generating-function expansion that yields higher-depth
congruences for $1\le n\le p-2$.

\begin{table}[H]
\centering
\tiny
\renewcommand{\arraystretch}{1.35}
\setlength{\tabcolsep}{8pt}

\begin{tabular}{p{0.46\textwidth} p{0.46\textwidth}}
\toprule
\multicolumn{1}{c}{\(\boldsymbol{p=3}\)}
&
\multicolumn{1}{c}{\(\boldsymbol{p=5}\)}
\\
\midrule

For \(n=4Q+r\), \(0\le r\le3\),
\[
v_3:=v_3\bigl(\G_n(3)\bigr)
=
r+v_3\binom nr.
\]
Equivalently,
\[
v_3   
=
\begin{cases}
0,
&n\equiv0\pmod4,\\
1+v_3(n),
&n\equiv1\pmod4,\\
2+v_3(n)+v_3(n-1),
&n\equiv2\pmod4,\\
2+v_3(n)+v_3(n-1)+v_3(n-2),
&n\equiv3\pmod4.
\end{cases}
\]
&
For \(n=4Q+r\), \(0\le r\le3\),
\[
v_5:=v_5\bigl(\G_n(5)\bigr)
=
r+v_5\binom nr.
\]
Equivalently,
\[
v_5
=
\begin{cases}
0,
&n\equiv0\pmod4,\\
1+v_5(n),
&n\equiv1\pmod4,\\
2+v_5(n)+v_5(n-1),
&n\equiv2\pmod4,\\
3+v_5(n)+v_5(n-1)+v_5(n-2),
&n\equiv3\pmod4.
\end{cases}
\]
\\

\midrule
\multicolumn{1}{c}{\(\boldsymbol{p\equiv3\pmod4}\)}
&
\multicolumn{1}{c}{\(\boldsymbol{p\equiv1\pmod4}\)}
\\
\midrule

\multicolumn{2}{c}{
\(\displaystyle
\text{For }n=4t+r\le p-2,\quad t\ge1,\quad r\in\{0,1,2,3\},
\)
}
\\[-1mm]

\multicolumn{2}{p{0.94\textwidth}}{
\[
\G_{4t+r}(p)
\equiv
\kappa_r(4t+r)_r\,p^{r+1}B_{4t}
\pmod{p^{r+2}},
\]
where
\[
\kappa_0=-2,\qquad
\kappa_1=-(1+\ii),\qquad
\kappa_2=-\frac{\ii}{2},\qquad
\kappa_3=\frac{1-\ii}{12}.
\]
Consequently,
\[
v_p\bigl(\G_{4t+r}(p)\bigr)\ge r+1,
\]
with equality if and only if
\[
p\nmid\operatorname{num}(B_{4t}).
\]
If \(p\mid\operatorname{num}(B_{4t})\), then
\[
v_p\bigl(\G_{4t+r}(p)\bigr)\ge r+2.
\]
}
\\

\midrule

\multicolumn{2}{p{0.94\textwidth}}{
For every \(p\ge5\), if
\(
n\equiv2\pmod4,
\)
then
\[
v_p(\G_n(p))\ge2.
\]
If, in addition,
\(
n\not\equiv2\pmod{p-1},
\)
then
\[
v_p(\G_n(p))\ge3.
\]
}
\\

\midrule

For \(n=r(p-1)\), \(r\ge1\),
\[
\G_{r(p-1)}(p)
\equiv
1+(-1)^r-(p+1)\mathbf 1_{p+1\mid r}
\pmod p.
\]
Thus
\[
\G_{r(p-1)}(p)\equiv
\begin{cases}
0\pmod p,
&r\text{ odd},\\
2\pmod p,
&r\text{ even},\ p+1\nmid r,\\
1\pmod p,
&p+1\mid r.
\end{cases}
\]
For \(p\ge7\),
\[
v_p\bigl(\G_{r(p-1)}(p)\bigr)
=
0
\qquad(r\text{ even}),
\]
whereas
\[
v_p\bigl(\G_{r(p-1)}(p)\bigr)
\ge3
\qquad(r\text{ odd}).
\]
&
For \(n=r(p-1)\), \(1\le r<p\),
\[
\G_{r(p-1)}(p)
\equiv
1+\ii^{r(p-1)}
=
2
\pmod p.
\]
Hence
\[
v_p\bigl(\G_{r(p-1)}(p)\bigr)=0
\qquad(1\le r<p).
\]
No corresponding formula for arbitrary \(r\ge p\) is proved in the
article.
\\

\midrule

For \(n=p\), \(p\ge7\),
\[
\G_p(p)
\equiv
-\frac{p^5}{12}(p-1)^2(p-2)
B_{p-3}(1-\ii)
\pmod{p^6}.
\]
Therefore
\[
v_p(\G_p(p))\ge5.
\]
&
For \(n=p\),
\[
\G_p(p)
\equiv
p^2(1+\ii)
\pmod{p^3}.
\]
Consequently,
\[
v_p(\G_p(p))=2.
\]
\\

\bottomrule
\end{tabular}

\caption{Summary of the main valuation and congruence results. The
first part records the complete formulas for \(p=3\) and \(p=5\).
The columns are then relabelled according to the inert and split cases
\(p\equiv3\pmod4\) and \(p\equiv1\pmod4\). Results independent of the
splitting behavior are displayed across both columns.}
\label{tab:summary-results}
\end{table}

\section{Definitions and an exact reduction to classical power sums}\label{sec_2}

\begin{definition}
For $n\in\ZZ_{\ge 0}$ and $k,m\in\ZZ_{\ge 1}$ define
\[
\F_n(k,m)=\sum_{a=1}^{k}\sum_{b=1}^{m}(a+b\ii)^n \in \ZZ[\ii].
\]
Also define the classical power sums
\[
S_r(N):=\sum_{t=1}^{N} t^r \in \ZZ \qquad (r\in\ZZ_{\ge 0},\ N\in\ZZ_{\ge 1}).
\]
\end{definition}

\begin{lemma}[Binomial--power-sum factorization]\label{lem:binomfactor}
For every $n\ge 0$ and $k,m\ge 1$,
\[
\F_n(k,m)=\sum_{j=0}^{n}\binom{n}{j}\,\ii^{\,j}\,S_{n-j}(k)\,S_j(m).
\]
\end{lemma}

\begin{proof}
Expand \((a+b\ii)^n=\sum_{j=0}^n \binom{n}{j}a^{n-j}(b\ii)^j\). Since all sums are finite,
\[
\F_n(k,m)=\sum_{a=1}^k\sum_{b=1}^m\sum_{j=0}^n \binom{n}{j}\ii^{\,j}a^{n-j}b^j
=\sum_{j=0}^n \binom{n}{j}\ii^{\,j}\Bigl(\sum_{a=1}^k a^{n-j}\Bigr)\Bigl(\sum_{b=1}^m b^j\Bigr),
\]
which is the claimed formula.
\end{proof}

\begin{proposition}[Polynomiality in $(k,m)$]
\label{prop:polynomiality}
For fixed $n$, the function $(k,m)\mapsto \F_n(k,m)$ agrees with a polynomial in $(k,m)$ with Gaussian rational
coefficients, of total degree $n+2$.
\end{proposition}

\begin{proof}
By Lemma~\ref{lem:binomfactor}, $\F_n(k,m)$ is a finite $\QQ(\ii)$-linear combination of products
$S_{n-j}(k)S_j(m)$. By Faulhaber's formula, each $S_r(N)$ is a polynomial in $N$ of degree $r+1$
with rational coefficients. Hence $S_{n-j}(k)S_j(m)$ has total degree at most $(n-j+1)+(j+1)=n+2$. Moreover, the 
$j=0$ term contributes $\frac{1}{n+1}k^{n+1}m$ to the top-degree part, so the total degree is exactly $n+2$.
\end{proof}

\subsection{A symmetry specific to the square case $k=m$}

For $N\ge 2$ define the square sum
\[
\G_n(N):=\F_n(N-1,N-1)=\sum_{a=1}^{N-1}\sum_{b=1}^{N-1}(a+b\ii)^n.
\]

\begin{lemma}[Square symmetry]\label{lem:squaresym}
For every $n\ge 0$ and $N\ge 1$,
\[
\G_n(N)=\ii^{\,n}\,\overline{\G_n(N)}.
\]
\end{lemma}

\begin{proof}
Because the domain is symmetric in $a$ and $b$,
\[
\G_n(N)=\sum_{a=1}^{N-1}\sum_{b=1}^{N-1}(b+a\ii)^n.
\]
But $b+a\ii=\ii(a-b\ii)$, hence $(b+a\ii)^n=\ii^{\,n}(a-b\ii)^n$. Therefore
\[
\G_n(N)=\ii^{\,n}\sum_{a=1}^{N-1}\sum_{b=1}^{N-1}(a-b\ii)^n=\ii^{\,n}\,\overline{\G_n(N)}.
\]
\end{proof}

\begin{corollary}[Exact location in the complex plane]\label{cor:location}
Write $\G_n(N)=x+y\ii$ with $x,y\in\ZZ$. Then:
\[
\begin{array}{rcl}
n\equiv 0\ (\mathrm{mod}\ 4) &\Rightarrow& y=0 \quad (\text{$\G_n(N)$ is real}),\\
n\equiv 2\ (\mathrm{mod}\ 4) &\Rightarrow& x=0 \quad (\text{$\G_n(N)$ is purely imaginary}),\\
n\equiv 1\ (\mathrm{mod}\ 4) &\Rightarrow& x=y,\\
n\equiv 3\ (\mathrm{mod}\ 4) &\Rightarrow& x=-y.
\end{array}
\]
\end{corollary}

\begin{proof}
From Lemma~\ref{lem:squaresym} we have $x+y\ii=\ii^{\,n}(x-y\ii)$. Comparing real and imaginary parts
in the four cases $n\ (\mathrm{mod}\ 4)$ gives the stated equalities.
\end{proof}

\section{Congruences mod $p$ and $p^2$}\label{sec_3}

Fix an odd prime $p$ and define the main object
\[
\G_n(p):=\F_n(p-1,p-1)=\sum_{a=1}^{p-1}\sum_{b=1}^{p-1}(a+b\ii)^n\in\ZZ[\ii].
\]
We also write $v_p(z)$ for the $p$-adic valuation of $z\in\ZZ$, and for $z=x+y\ii\in\ZZ[\ii]$ we set
\[
v_p(z):=\min\{v_p(x),v_p(y)\}.
\]

\subsection{Two standard lemmas}

\begin{lemma}[Power sums modulo \(p\)]\label{lem:powersum_modp}
Let \(p\) be prime and \(t\ge 0\). Then
\[
\sum_{a=1}^{p-1} a^t \equiv
\begin{cases}
-1 \pmod p, & (p-1)\mid t,\\
0 \pmod p,  & (p-1)\nmid t.
\end{cases}
\]
\end{lemma}

\begin{proof}
If \((p-1)\mid t\), then \(a^t\equiv 1\pmod p\) for all \(a\in \mathbb{F}_p^\times\), so
\[
\sum_{a=1}^{p-1} a^t \equiv \sum_{a=1}^{p-1} 1 = p-1 \equiv -1 \pmod p.
\]
If \((p-1)\nmid t\), fix a generator \(g\) of the cyclic group \(\mathbb{F}_p^\times\). Then
\[
\sum_{a=1}^{p-1} a^t
= \sum_{u=0}^{p-2} (g^u)^t
= \sum_{u=0}^{p-2} (g^t)^u.
\]
Since \((p-1)\nmid t\), we have \(g^t\neq 1\) in \(\mathbb{F}_p^\times\), hence this is a geometric series with sum
\[
\sum_{u=0}^{p-2} (g^t)^u = \frac{(g^t)^{p-1}-1}{g^t-1}=0
\]
in \(\mathbb{F}_p\), which yields the desired congruence.
\end{proof}

\begin{lemma}[Lucas--vanishing for stretched binomial coefficients]\label{lem:binom_lucas}
Let \(p\) be prime and \(1\le r<p\). Then for every integer \(t\) with \(0<t<r\),
\[
\binom{r(p-1)}{t(p-1)} \equiv 0 \pmod p.
\]
Moreover,
\[
\binom{r(p-1)}{0}\equiv 1 \pmod p,
\qquad
\binom{r(p-1)}{r(p-1)}\equiv 1 \pmod p.
\]
\end{lemma}

\begin{proof}
Expand the numbers in base \(p\):
\[
r(p-1)=rp-r=(r-1)p+(p-r),
\qquad
t(p-1)=tp-t=(t-1)p+(p-t),
\]
where \(0\le p-r\le p-1\) and \(0\le p-t\le p-1\). Since \(0<t<r\), we have \(p-t>p-r\).

By Lucas' theorem,
\[
\binom{r(p-1)}{t(p-1)}
\equiv
\binom{r-1}{t-1}\binom{p-r}{p-t}
\pmod p.
\]
But \(\binom{p-r}{p-t}=0\) as an integer because \(p-t>p-r\). Hence
\[
\binom{r(p-1)}{t(p-1)} \equiv 0 \pmod p.
\]
The endpoint congruences are immediate.
\end{proof}

\subsection{The reduction mod \(p\)}

\begin{theorem}[A mod \(p\) dichotomy for \(\G_n(p)\)]\label{thm:modp}
Let \(p\) be an odd prime.
\begin{enumerate}
\item If \((p-1)\nmid n\), then \(\G_n(p)\equiv 0\pmod p\) in \(\mathbb{Z}[\ii]\).
\item If \(n=r(p-1)\) with \(1\le r<p\), then
\[
\G_n(p)\equiv 1+\ii^{\,r(p-1)}\pmod p.
\]
In particular, \(\G_{p-1}(p)\equiv 0\pmod p\) if \(p\equiv 3\pmod 4\) and \(\G_{p-1}(p)\equiv 2\pmod p\)
if \(p\equiv 1\pmod 4\).
\end{enumerate}
\end{theorem}

\begin{proof}
Use Lemma~\ref{lem:binomfactor} with \(k=m=p-1\):
\[
\G_n(p)=\sum_{j=0}^{n}\binom{n}{j}\ii^{\,j}
\Big(\sum_{a=1}^{p-1} a^{n-j}\Big)
\Big(\sum_{b=1}^{p-1} b^{j}\Big).
\]
If \((p-1)\nmid n\), then for every \(j\) at least one of \(j\) or \(n-j\) is not divisible by \(p-1\).
By Lemma~\ref{lem:powersum_modp}, at least one of the two power sums is \(0\pmod p\), hence \(\G_n(p)\equiv 0\pmod p\).

Now suppose \(n=r(p-1)\) with \(1\le r<p\).

By Lemma~\ref{lem:powersum_modp}, the only terms in \(\G_n(p)\) that are non-zero are those for which both  \(j\) and \(n-j\) are divisible by \(p-1\). Let us focus on those. For such terms, we have \(j=t(p-1)\) for some \(t\in\{0,1,\dots,r\}\) since \(j\le n=r(p-1)\).
Thus
\[
\G_n(p)\equiv \sum_{t=0}^{r}\binom{r(p-1)}{t(p-1)}\ii^{\,t(p-1)}\cdot (-1)\cdot (-1)\pmod p.
\]
By Lemma~\ref{lem:binom_lucas}, all intermediate coefficients vanish modulo \(p\) for \(0<t<r\), leaving only
\(t=0\) and \(t=r\), and therefore
\[
\G_n(p)\equiv 1+\ii^{\,r(p-1)}\pmod p.
\]

The result for the special case where  \(n=p-1\) and  \(p\equiv 1\pmod 4\) follows from the fact that \(\ii^{\,p-1}=1\). The result for the other case also follows because \(\ii^{\,p-1}=-1\) if \(p\equiv 3\pmod 4\).
\end{proof}

\begin{remark}
The restriction $1\le r<p$ in Theorem~\ref{thm:modp} is needed only
for the Lucas-theorem argument. For inert primes and arbitrary
$r\ge1$, the exact reduction of $\G_{r(p-1)}(p)$ is given by
Theorem~\ref{thm:multiples-p-minus-one-inert}.
\end{remark}

\subsection{The reduction mod \(p^2\) for \(1\le n\le p-2\)}

\begin{theorem}[Endpoint reduction mod \(p^2\)]\label{thm:modp2_endpoint}
Let \(p\) be an odd prime and let \(1\le n\le p-2\). Then
\[
\G_n(p)\equiv (p-1)\Big(\sum_{a=1}^{p-1}a^n\Big)\,(1+\ii^{\,n})\pmod{p^2}
\qquad \text{in }\mathbb{Z}[\ii].
\]
\end{theorem}

\begin{proof}
Start from Lemma~\ref{lem:binomfactor} with \(k=m=p-1\):
\[
\G_n(p)=\sum_{j=0}^{n}\binom{n}{j}\ii^{\,j}\,S_{n-j}(p-1)\,S_j(p-1).
\]
For \(1\le j\le n-1\), we have \(1\le j\le p-3\) and \(1\le n-j\le p-3\), hence \((p-1)\nmid j\) and \((p-1)\nmid(n-j)\).
By Lemma~\ref{lem:powersum_modp}, both \(S_j(p-1)\) and \(S_{n-j}(p-1)\) are divisible by \(p\), so their product
is divisible by \(p^2\). Hence all terms with \(1\le j\le n-1\) vanish modulo \(p^2\).

Thus, modulo \(p^2\) only the endpoints \(j=0\) and \(j=n\) remain:

\[
\G_n(p)\equiv \binom{n}{0}\ii^0 S_n(p-1)S_0(p-1)+\binom{n}{n}\ii^n S_0(p-1)S_n(p-1)\pmod{p^2}.
\]
Since \(S_0(p-1)=p-1\), this is exactly the claimed formula.
\end{proof}

\begin{corollary}[Even exponents and Bernoulli numbers]\label{cor:evenBn}
Let \(p\) be an odd prime and let \(n\) be even with \(2\le n\le p-3\).
Then, writing \(B_n\) for the Bernoulli number,
\[
\G_n(p)\equiv p(p-1)(1+\ii^{\,n})B_n\pmod{p^2}.
\]
In particular,
\[
n\equiv 2 \ (\mathrm{mod}\ 4)\ \Longrightarrow\ \G_n(p)\equiv 0\pmod{p^2},
\qquad
n\equiv 0 \ (\mathrm{mod}\ 4)\ \Longrightarrow\ \G_n(p)\equiv 2p(p-1)B_n\pmod{p^2}.
\]
\end{corollary}

\begin{proof}
Let $p$ be an odd prime and let $n$ be even with $2\le n\le p-3$.
Faulhaber's formula in Bernoulli-polynomial form gives
\[
S_n(p-1)=\sum_{a=1}^{p-1}a^n=\frac{B_{n+1}(p)-B_{n+1}}{n+1}.
\]
Since $n+1$ is odd and $>1$, we have $B_{n+1}=0$. Expanding $B_{n+1}(p)$ yields
\[
B_{n+1}(p)=\sum_{k=0}^{n+1}\binom{n+1}{k}B_{n+1-k}\,p^k,
\]
hence
\[
S_n(p-1)=\sum_{k=1}^{n+1}\binom{n}{k-1}\frac{p^k}{k}\,B_{n+1-k}.
\]
For $k\ge 2$ the factor $p^k$ contributes $p^2$. Moreover, for each index $n+1-k\le n-1\le p-4$,
von Staudt--Clausen implies $p\nmid \den(B_{n+1-k})$ (since $(p-1)\nmid (n+1-k)$ and if $k=n+1$ then $B_0=1$), so these terms
are indeed $0\pmod{p^2}$ in $\ZZ_{(p)}$. Therefore only $k=1$ survives modulo $p^2$, giving
\[
S_n(p-1)\equiv pB_n \pmod{p^2}.
\]
Substitute this into Theorem~\ref{thm:modp2_endpoint}.
\end{proof}

\begin{proposition}[The first three exponents]\label{prop:G123}
For every odd prime $p$ one has
\[
\G_1(p)=\frac{p(p-1)^2}{2}(1+\ii),
\]
\[
\G_2(p)=\frac{p^2(p-1)^2}{2}\,\ii,
\]
and
\[
\G_3(p)=-\frac{p^3(p-1)^2}{4}(1-\ii).
\]
In particular,
\[
v_p(\G_1(p))=1,\qquad
v_p(\G_2(p))=2,\qquad
v_p(\G_3(p))=3.
\]
\end{proposition}

\begin{proof}
Write
\[
S_r=S_r(p-1)=\sum_{a=1}^{p-1}a^r.
\]
We shall use
\[
S_0=p-1,\qquad
S_1=\frac{p(p-1)}{2},\qquad
S_2=\frac{p(p-1)(2p-1)}{6},\qquad
S_3=\left(\frac{p(p-1)}{2}\right)^2.
\]

For $n=1$,
\[
\G_1(p)=\sum_{a,b=1}^{p-1}(a+b\ii)
=S_1S_0+\ii S_0S_1
=S_0S_1(1+\ii),
\]
and therefore
\[
\G_1(p)=\frac{p(p-1)^2}{2}(1+\ii).
\]

For $n=2$,
\[
(a+b\ii)^2=a^2-b^2+2ab\ii.
\]
Hence
\[
\G_2(p)=S_2S_0-S_0S_2+2\ii S_1^2
=2\ii S_1^2
=\frac{p^2(p-1)^2}{2}\,\ii.
\]

For $n=3$,
\[
(a+b\ii)^3=a^3+3a^2b\ii-3ab^2-b^3\ii.
\]
Thus
\[
\G_3(p)=S_3S_0-3S_1S_2
+\ii(3S_2S_1-S_0S_3).
\]
Substituting the above formulas for $S_0,S_1,S_2,S_3$, we get
\[
S_3S_0-3S_1S_2
=
\frac{p^2(p-1)^3}{4}
-\frac{p^2(p-1)^2(2p-1)}{4}
=
-\frac{p^3(p-1)^2}{4}.
\]
Similarly,
\[
3S_2S_1-S_0S_3
=
\frac{p^3(p-1)^2}{4}.
\]
Therefore
\[
\G_3(p)
=
-\frac{p^3(p-1)^2}{4}
+\frac{p^3(p-1)^2}{4}\ii
=
-\frac{p^3(p-1)^2}{4}(1-\ii).
\]

Since $p$ is odd, the rational factors $1/2$ and $1/4$ are $p$-adic
units. The displayed formulas immediately imply
\[
v_p(\G_1(p))=1,\qquad
v_p(\G_2(p))=2,\qquad
v_p(\G_3(p))=3.
\]
\end{proof}

\subsection{The small-exponent valuation pattern and irregular pairs}
\begin{lemma}[A universal $p$-adic generating-function expansion]
\label{lem:master-p-expansion}
Put
\[
Q(z):=\frac{1}{e^z-1},
\qquad
C_\ell(z):=\frac{z^\ell}{\ell!}Q(z)
\qquad (\ell\ge1).
\]
Although $Q(z)$ is a Laurent series,
\[
Q(z)=z^{-1}-\frac12+O(z),
\]
one has
\[
C_\ell(z)\in z^{\ell-1}\QQ[[z]],
\]
so all coefficient extractions below are well defined in the ring of
formal power series.
For $\ell\ge1$, define
\[
D_\ell(z):=
-C_\ell(z)-C_\ell(\ii z)
+\sum_{a=1}^{\ell-1}C_a(z)C_{\ell-a}(\ii z).
\]
Then
\[
\sum_{n\ge0}\G_n(p)\frac{z^n}{n!}
=
1+\sum_{\ell\ge1}p^\ell D_\ell(z).
\]
Consequently, if $1\le n\le p-2$ and $M\ge1$, then
\[
\G_n(p)
\equiv
n!\sum_{\ell=1}^{M}p^\ell[z^n]D_\ell(z)
\pmod{p^{M+1}}
\]
in $\ZZ_{(p)}[\ii]$.
\end{lemma}

\begin{proof}
We have
\[
\sum_{m\ge0}S_m(p-1)\frac{z^m}{m!}
=
\sum_{a=1}^{p-1}e^{az}
=
\frac{e^{pz}-e^z}{e^z-1}.
\]
Since
\[
e^{pz}=1+\sum_{\ell\ge1}\frac{p^\ell z^\ell}{\ell!},
\]
it follows that
\[
\frac{e^{pz}-e^z}{e^z-1}
=
-1+\sum_{\ell\ge1}p^\ell C_\ell(z).
\]
Therefore
\begin{align*}
\sum_{n\ge0}\G_n(p)\frac{z^n}{n!}
&=
\left(-1+\sum_{\ell\ge1}p^\ell C_\ell(z)\right)
\left(-1+\sum_{\ell\ge1}p^\ell C_\ell(\ii z)\right)\\
&=
1+\sum_{\ell\ge1}p^\ell D_\ell(z).
\end{align*}
This proves the exact formal identity.

For fixed $n\le p-2$, the coefficient $[z^n]D_\ell(z)$ is a finite
$\QQ(\ii)$-linear combination of products of Bernoulli numbers whose
indices are at most $n$. Hence every Bernoulli index occurring is
strictly smaller than $p-1$.
Hence their denominators are prime to $p$ by von Staudt--Clausen.
Thus the terms with $\ell\ge M+1$ are divisible by $p^{M+1}$ in
$\ZZ_{(p)}[\ii]$, proving the congruence.
\end{proof}

\begin{lemma}[The first nonzero coefficients in each class modulo $4$]
\label{lem:D-coefficients}
Let $t\ge1$. Then
\[
[z^{4t}]D_1(z)=-\frac{2B_{4t}}{(4t)!},
\]
\[
[z^{4t+1}]D_1(z)=0,
\qquad
[z^{4t+1}]D_2(z)
=
-\frac{(4t+1)(1+\ii)B_{4t}}{(4t+1)!},
\]
\[
[z^{4t+2}]D_1(z)=[z^{4t+2}]D_2(z)=0,
\]
\[
[z^{4t+2}]D_3(z)
=
-\frac{(4t+2)(4t+1)\ii B_{4t}}
       {2(4t+2)!},
\]
and
\[
[z^{4t+3}]D_1(z)
=
[z^{4t+3}]D_2(z)
=
[z^{4t+3}]D_3(z)
=
0,
\]
while
\[
[z^{4t+3}]D_4(z)
=
\frac{(4t+3)(4t+2)(4t+1)(1-\ii)B_{4t}}
     {12(4t+3)!}.
\]
\end{lemma}\begin{proof}
Write
\[
Q=Q(z)=\frac1{e^z-1},
\qquad
Q_{\ii}=Q(\ii z)=\frac1{e^{\ii z}-1}.
\]
From the definition one obtains
\[
D_1(z)=-z(Q+\ii Q_{\ii}),
\]
\[
D_2(z)
=
-\frac{z^2}{2}(Q-Q_{\ii})
+\ii z^2QQ_{\ii},
\]
\[
D_3(z)
=
-\frac{z^3}{6}(Q-\ii Q_{\ii})
+\frac{-1+\ii}{2}z^3QQ_{\ii},
\]
and
\[
D_4(z)
=
-\frac{z^4}{24}(Q+Q_{\ii})
-\frac{z^4}{4}QQ_{\ii}.
\]
Now use
\[
Q(z)=\sum_{m\ge0}B_m\frac{z^{m-1}}{m!},
\qquad
Q(\ii z)=
\sum_{m\ge0}B_m\ii^{m-1}\frac{z^{m-1}}{m!},
\]
together with $B_m=0$ for odd $m>1$.

For later use, write
\[
[z^{N-2}]QQ_{\ii}
=
\sum_{a=0}^{N}
\frac{B_aB_{N-a}}{a!(N-a)!}\ii^{\,N-a-1}.
\]
Because $B_m=0$ for odd $m>1$, only the indices
$a\in\{0,1,N-1,N\}$ and the even indices contribute.  For the values
of $N$ occurring below, the factors $\ii^{\,N-a-1}$ depend only on
$a\bmod4$. Separating the even indices into the two classes modulo $4$
and using
\[
\sum_{a=0}^{N}\binom NaB_aB_{N-a}
=-(N-1)B_N
\qquad(N\ge4,\ N\text{ even}),
\]
one obtains
\[
(4t)![z^{4t-1}]QQ_{\ii}
=-2t(1+\ii)B_{4t},
\]
\[
(4t)![z^{4t}]QQ_{\ii}
=-\frac{4t+1}{2}\ii B_{4t},
\]
and
\[
(4t)![z^{4t+1}]QQ_{\ii}
=-\frac{(4t+1)(4t+2)}{6}(1-\ii)B_{4t}.
\]
Combining these identities with the contributions from the linear
terms involving $Q$ and $Q_{\ii}$ gives the asserted coefficients of
$D_1,D_2,D_3,D_4$.
\end{proof}

\begin{theorem}[Uniform small-exponent expansion]
\label{thm:small-n-bernoulli}
Let $p$ be an odd prime, let $t\ge1$, and let
$r\in\{0,1,2,3\}$ satisfy
\[
4t+r\le p-2.
\]
Put
\[
\kappa_0=-2,\qquad
\kappa_1=-(1+\ii),\qquad
\kappa_2=-\frac{\ii}{2},\qquad
\kappa_3=\frac{1-\ii}{12},
\]
and write
\[
(n)_0:=1,\qquad
(n)_r:=n(n-1)\cdots(n-r+1)
\quad(r\ge1).
\]
Then
\begin{equation}\label{eq:uniform-small-expansion}
\G_{4t+r}(p)
\equiv
\kappa_r(4t+r)_r\,p^{r+1}B_{4t}
\pmod{p^{r+2}}.
\end{equation}

Consequently,
\[
v_p(\G_{4t+r}(p))\ge r+1.
\]
Moreover,
\[
v_p(\G_{4t+r}(p))=r+1
\quad\Longleftrightarrow\quad
p\nmid\operatorname{num}(B_{4t}),
\]
and
\[
p\mid\operatorname{num}(B_{4t})
\quad\Longleftrightarrow\quad
v_p(\G_{4t+r}(p))\ge r+2.
\]
Thus an irregular pair $(p,4t)$ produces an exceptional block whose
four valuations are at least
\[
(2,3,4,5).
\]
\end{theorem}

\begin{proof}
Set $n=4t+r$. By Lemma~\ref{lem:master-p-expansion},
\[
\G_n(p)
\equiv
n!\sum_{\ell=1}^{r+1}
p^\ell[z^n]D_\ell(z)
\pmod{p^{r+2}}.
\]
Lemma~\ref{lem:D-coefficients} shows that
\[
[z^n]D_\ell(z)=0
\qquad(1\le\ell\le r),
\]
while
\[
n![z^n]D_{r+1}(z)
=
\kappa_r(n)_rB_{4t}.
\]
This proves \eqref{eq:uniform-small-expansion}.

Since $4t<p-1$, the von Staudt--Clausen theorem implies
\[
p\nmid\den(B_{4t}).
\]
Also, under the hypothesis $4t+r\le p-2$, every factor in
$(4t+r)_r$ is a positive integer smaller than $p$. The denominators
of the constants $\kappa_r$ are $p$-adic units whenever the
corresponding range is nonempty. Therefore
\[
\kappa_r(4t+r)_r
\]
is a $p$-adic unit, and the valuation assertions follow.
\end{proof}

\subsection{A general depth-$3$ criterion and multiples of $p-1$}

We first record a divisibility property of ordinary power sums.

\begin{lemma}\label{lem:odd-power-sums-mod-p2}
Let \(p\ge5\) be prime and put
\[
S_m(p):=\sum_{a=1}^{p-1}a^m.
\]
If \(m\ge1\) is odd, then
\[
v_p\bigl(S_m(p)\bigr)\ge1.
\]
Moreover,
\[
v_p\bigl(S_m(p)\bigr)\ge2
\]
unless
\[
m\equiv1\pmod{p-1}.
\]
More precisely,
\[
\frac{S_m(p)}p
\equiv
\begin{cases}
-\dfrac{m}{2}\pmod p,
& m\equiv1\pmod{p-1},\\[2mm]
0\pmod p,
& m\not\equiv1\pmod{p-1}.
\end{cases}
\]
\end{lemma}

\begin{proof}
Faulhaber's formula gives
\[
S_m(p)
=
\frac{B_{m+1}(p)-B_{m+1}}{m+1}
=
\sum_{j=1}^{m+1}
\frac1{m+1}\binom{m+1}{j}B_{m+1-j}p^j.
\]
Since \(m\) is odd, \(m+1\) is even. The term with \(j=1\) is
\[
B_m p.
\]
For \(m>1\), we have \(B_m=0\), while for \(m=1\) the same conclusion
below follows directly from
\[
S_1(p)=\frac{p(p-1)}2.
\]

For $m>1$, the term with $j=1$ vanishes because $B_m=0$.  Using
\[
\frac{1}{m+1}\binom{m+1}{j}
=
\frac1j\binom{m}{j-1},
\]
the term with index $j$ may be written as
\[
\frac{p^j}{j}\binom{m}{j-1}B_{m+1-j}.
\]
By the von Staudt--Clausen theorem,
\[
v_p(B_{m+1-j})\ge-1.
\]
Hence, for every $j\ge3$,
\[
v_p\!\left(
\frac{p^j}{j}\binom{m}{j-1}B_{m+1-j}
\right)
\ge
j-v_p(j)-1
\ge2,
\]
because $p\ge5$. Therefore
\[
S_m(p)
\equiv
\frac m2 B_{m-1}p^2
\pmod{p^2\ZZ_{(p)}}.
\]
Here the displayed main term need not itself lie in
$p^2\ZZ_{(p)}$, since $v_p(B_{m-1})$ may equal $-1$.

If $p-1\nmid m-1$, then $B_{m-1}\in\ZZ_{(p)}$, and consequently
\[
S_m(p)\equiv0\pmod{p^2}.
\]
If $p-1\mid m-1$, then the von Staudt--Clausen theorem gives
\[
B_{m-1}=-\frac1p+u,
\qquad
u\in\ZZ_{(p)}.
\]
Thus
\[
S_m(p)\equiv-\frac m2p\pmod{p^2}.
\]
For $m=1$,
\[
S_1(p)=\frac{p(p-1)}2\equiv-\frac p2\pmod{p^2}.
\]
Dividing by $p$ proves the stated formula.\end{proof}

\begin{theorem}[A general depth-$3$ criterion]
\label{thm:general-depth-three}
Let $p\ge5$ be prime and let $n\ge1$ satisfy
\[
n\equiv2\pmod4.
\]
Then
\[
v_p(\G_n(p))\ge2.
\]
If, in addition,
\[
n\not\equiv2\pmod{p-1},
\]
then
\[
v_p(\G_n(p))\ge3.
\]
\end{theorem}

\begin{proof}
From Lemma~\ref{lem:binomfactor},
\[
\G_n(p)
=
\sum_{k=0}^{n}
\binom nk\ii^k
S_k(p-1)S_{n-k}(p-1).
\]
Pair the terms indexed by $k$ and $n-k$. Since $n\equiv2\pmod4$,
we have $\ii^n=-1$. If $k$ is even, then
\[
\ii^{n-k}
=
\ii^n\ii^{-k}
=
-\ii^{-k}
=
-\ii^k,
\]
so the paired terms cancel. Therefore
\begin{equation}\label{eq:only-odd-general}
\G_n(p)
=
\sum_{\substack{0\le k\le n\\k\text{ odd}}}
\binom nk\ii^k
S_k(p-1)S_{n-k}(p-1).
\end{equation}

For odd $k$, both $k$ and $n-k$ are odd. By
Lemma~\ref{lem:odd-power-sums-mod-p2}, both corresponding power sums
are divisible by $p$. Hence every term in
\eqref{eq:only-odd-general} is divisible by $p^2$.

Both power sums can have valuation exactly $1$ only if
\[
k\equiv1\pmod{p-1},
\qquad
n-k\equiv1\pmod{p-1}.
\]
These two congruences imply
\[
n\equiv2\pmod{p-1}.
\]
Therefore, if $n\not\equiv2\pmod{p-1}$, at least one of the two power
sums in every summand is divisible by $p^2$. Every summand is then
divisible by $p^3$.
\end{proof}
\begin{theorem}[Exact reduction at multiples of $p-1$]
\label{thm:multiples-p-minus-one-inert}
Let $p\equiv3\pmod4$ be prime and let $r\ge1$. Then
\begin{equation}\label{eq:exact-multiple-mod-p}
\G_{r(p-1)}(p)
\equiv
1+(-1)^r-(p+1)\mathbf 1_{p+1\mid r}
\pmod p,
\end{equation}
where $\mathbf 1_{p+1\mid r}$ is $1$ if $p+1\mid r$ and $0$
otherwise. Equivalently,
\[
\G_{r(p-1)}(p)\equiv
\begin{cases}
0\pmod p,
&r\text{ odd},\\[1mm]
2\pmod p,
&r\text{ even and }p+1\nmid r,\\[1mm]
1\pmod p,
&p+1\mid r.
\end{cases}
\]

If $p\ge7$, then
\[
v_p\bigl(\G_{r(p-1)}(p)\bigr)=0
\qquad\text{if $r$ is even},
\]
whereas
\[
v_p\bigl(\G_{r(p-1)}(p)\bigr)\ge3
\qquad\text{if $r$ is odd}.
\]
Consequently, among the positive multiples of $p-1$,
\[
v_p(\G_n(p))=0
\quad\Longleftrightarrow\quad
\operatorname{lcm}(4,p-1)\mid n.
\]
\end{theorem}

\begin{proof}
Work in
\[
\FF_{p^2}=\FF_p(\ii),
\qquad
\ii^p=-\ii.
\]
For $a,b\in\FF_p^\times$, write $a=tb$. Since
\[
n=r(p-1),
\]
we have $b^n=1$, and therefore
\begin{align}
\G_{r(p-1)}(p)
&\equiv
\sum_{b\in\FF_p^\times}
\sum_{t\in\FF_p^\times}
b^{r(p-1)}(t+\ii)^{r(p-1)}
\pmod p
\nonumber\\
&\equiv
-\sum_{t\in\FF_p^\times}
(t+\ii)^{r(p-1)}
\pmod p.
\label{eq:G-mod-p-first}
\end{align}

For $t\in\FF_p$,
\[
(t+\ii)^{p-1}
=
\frac{(t+\ii)^p}{t+\ii}
=
\frac{t-\ii}{t+\ii}.
\]
Let
\[
U:=\{u\in\FF_{p^2}^{\times}:u^{p+1}=1\}.
\]
The map
\[
\mathbb P^1(\FF_p)\longrightarrow U,
\qquad
t\longmapsto\frac{t-\ii}{t+\ii},
\]
with value $1$ at $t=\infty$, is a bijection. It sends
\[
0\longmapsto-1,
\qquad
\infty\longmapsto1.
\]
Hence
\[
\sum_{t\in\FF_p^\times}
\left(\frac{t-\ii}{t+\ii}\right)^r
=
\sum_{u\in U}u^r-1-(-1)^r.
\]
Substituting into \eqref{eq:G-mod-p-first}, we obtain
\[
\G_{r(p-1)}(p)
\equiv
-\sum_{u\in U}u^r+1+(-1)^r
\pmod p.
\]

Since $U$ is cyclic of order $p+1$,
\[
\sum_{u\in U}u^r
=
\begin{cases}
0,
&p+1\nmid r,\\
p+1,
&p+1\mid r.
\end{cases}
\]
This proves \eqref{eq:exact-multiple-mod-p} and its three-case
reformulation.

In particular, if $r$ is even, then
\[
\G_{r(p-1)}(p)\equiv1\text{ or }2\pmod p,
\]
and therefore
\[
v_p\bigl(\G_{r(p-1)}(p)\bigr)=0.
\]

Now suppose that $p\ge7$ and $r$ is odd. Since
\[
p-1\equiv2\pmod4,
\]
we have
\[
r(p-1)\equiv2\pmod4.
\]
Moreover,
\[
r(p-1)\equiv0\pmod{p-1},
\]
and, because $p-1\ge6$,
\[
0\not\equiv2\pmod{p-1}.
\]
Therefore Theorem~\ref{thm:general-depth-three}, applied with
$n=r(p-1)$, gives
\[
v_p\bigl(\G_{r(p-1)}(p)\bigr)\ge3.
\]

Finally, since $p-1\equiv2\pmod4$,
\[
\operatorname{lcm}(4,p-1)\mid r(p-1)
\quad\Longleftrightarrow\quad
2\mid r.
\]
The final assertion follows.
\end{proof}

\begin{remark}\label{rem:valuation-not-always-three}
The lower bound in Theorem~\ref{thm:multiples-p-minus-one-inert}
cannot in general be replaced by equality. For example, direct
calculation gives
\[
v_7\bigl(\G_{30}(7)\bigr)=5,
\]
where
\[
30=5(7-1).
\]
Thus, for odd \(r\), the valuation is always at least \(3\), but it may
increase at exceptional values of \(r\).
\end{remark}

\section{The special exponent $n=p$: Bernoulli truncations modulo $p^6$}

Before proving the theorems we need to recall some facts about Bernoulli numbers and Bernoulli polynomials. Let $B_n(x)$ be the Bernoulli polynomials and $B_n:=B_n(0)$ the Bernoulli numbers, with $B_1=-\tfrac12$.
Recall the identity

\[
B_n(x)=\sum_{j=0}^{n}\binom{n}{j}B_{n-j}\,x^j \qquad (n\ge 0).
\]
and Faulhaber's formula
\[
S_r(p-1)=\sum_{t=1}^{p-1}t^r=\frac{B_{r+1}(p)-B_{r+1}}{r+1}.
\]
Expanding $B_{r+1}(p)$ by the Bernoulli-polynomial identity and cancelling the constant term gives the \emph{exact}
identity in $\QQ$:
\begin{equation}\label{eq:faulhaber-p-adic}
S_r(p-1)=\sum_{t=1}^{r+1}\frac{p^t}{t}\binom{r}{t-1}B_{r-t+1}\qquad (r\ge 0).
\end{equation}
 
Throughout this section, congruences are interpreted $p$-adically. Whenever Bernoulli numbers with $p$-integral denominator occur, we work in $\ZZ_{(p)}$. The exceptional term $B_{p-1}$ has $p$-adic valuation $-1$ by the von Staudt--Clausen theorem, but it always occurs multiplied by a sufficient power of $p$ so that the resulting expressions are $p$-integral.
Using \eqref{eq:faulhaber-p-adic} and the standard fact $B_n=0$ for odd $n>1$, we obtain:

\begin{lemma}[Safe $p^6$-truncations for $p\ge 7$ and $r\le p$]\label{lem:Sr-trunc}
Let $p\ge 7$ be prime and write $S_r=S_r(p-1)=\sum_{t=1}^{p-1}t^r$ with $r\le p$.
Then in $\ZZ_{(p)}$:
\begin{enumerate}
\item If $r\ge 6$ is even, then
\begin{equation}\label{eq:Sr-even}
S_r \equiv
pB_r+\frac{p^3}{3!}\,r(r-1)B_{r-2}+\frac{p^5}{5!}\,r(r-1)(r-2)(r-3)B_{r-4}
\pmod{p^6}.
\end{equation}
For $r\in\{2,4\}$ the same congruence holds after adding the extra term $p^rB_1$.
\item If $r\ge 7$ is odd, then
\begin{equation}\label{eq:Sr-odd}
S_r \equiv
\frac{p^2}{2!}\,r\,B_{r-1}+\frac{p^4}{4!}\,r(r-1)(r-2)B_{r-3}
\pmod{p^6}.
\end{equation}
For $r\in\{3,5\}$ the same congruence holds after adding the extra term $p^rB_1$.
\item For $r=1$ we have the exact identity
\begin{equation}\label{eq:S1-exact}
S_1(p-1)=\frac{p(p-1)}{2}.
\end{equation}
\end{enumerate}
\end{lemma}

\begin{proof}
We use the exact identity \eqref{eq:faulhaber-p-adic}:
\[
S_r(p-1)
=
\sum_{t=1}^{r+1}
\frac{p^t}{t}\binom{r}{t-1}B_{r-t+1}.
\]

By the von Staudt--Clausen theorem, every nonzero Bernoulli number has
square-free denominator.  Moreover, $B_m=0$ for every odd $m>1$.

Suppose first that $r$ is even.  Apart from the term containing $B_1$,
which occurs when $t=r$, all even values of $t$ give the odd Bernoulli
index $r-t+1>1$ and therefore vanish.  The nonzero terms with
$t\le5$ are consequently those with $t=1,3,5$, giving
\[
pB_r,\qquad
\frac{p^3}{3!}r(r-1)B_{r-2},\qquad
\frac{p^5}{5!}r(r-1)(r-2)(r-3)B_{r-4}.
\]
The $B_1$-term equals $p^rB_1$ and must be retained only for
$r=2,4$.

Suppose now that $r$ is odd. Apart from the $B_1$-term at $t=r$, all
odd values of $t$ give an odd Bernoulli index greater than $1$ and
therefore vanish. The only nonzero terms with $t\le5$ are those with
$t=2,4$, giving
\[
\frac{p^2}{2!}rB_{r-1},
\qquad
\frac{p^4}{4!}r(r-1)(r-2)B_{r-3}.
\]
Again, the $B_1$-term $p^rB_1$ must be retained only for $r=3,5$.

It remains to justify that all terms with $t\ge6$ vanish modulo
$p^6$. If $6\le t\le p-1$, then $t$ is a $p$-adic unit and
$r-t+1\le p-5$, so $B_{r-t+1}$ has $p$-integral denominator.
Hence the corresponding term lies in $p^t\ZZ_{(p)}\subseteq
p^6\ZZ_{(p)}$.

Since $1\le t\le r+1\le p+1$, the only value of $t$ in this range
that is divisible by $p$ is $t=p$. In that case
\[
\frac{p^t}{t}=p^{p-1}\in p^6\ZZ_{(p)}
\]
because $p\ge7$. The value $t=p+1$ can occur only when $r=p$; then
$t$ is a $p$-adic unit and the Bernoulli index is $0$. This proves the asserted truncations.
The formula for $S_1(p-1)$ is elementary.
\end{proof}

We will apply Lemma~\ref{lem:Sr-trunc} with $r\in\{p-3,p-2,p-1,p\}$.
Note that $p-1$ is even and $p$ is odd, so \eqref{eq:Sr-even} applies to $S_{p-1}(p-1)$ and \eqref{eq:Sr-odd} applies to $S_p(p-1)$.

Write $S_r:=S_r(p-1)$ for brevity. By Lemma~\ref{lem:binomfactor},
\begin{equation}\label{eq:Gp-binomial}
\G_p(p)=\sum_{j=0}^{p}\binom{p}{j}\ii^{\,j}\,S_{p-j}\,S_j.
\end{equation}
We split \eqref{eq:Gp-binomial} into endpoints, near-endpoints, and the bulk:
\[
\G_p(p)
=
\underbrace{\bigl(\binom{p}{0}\ii^0S_pS_0+\binom{p}{p}\ii^pS_0S_p\bigr)}_{\text{(I) }j=0,p}
+
\underbrace{\bigl(\binom{p}{1}\ii^1S_{p-1}S_1+\binom{p}{p-1}\ii^{p-1}S_1S_{p-1}\bigr)}_{\text{(II) }j=1,p-1}
+
\underbrace{\sum_{j=2}^{p-2}\binom{p}{j}\ii^{\,j}S_{p-j}S_j}_{\text{(III) }2\le j\le p-2}.
\]
Since $S_0=p-1$, parts (I) and (II) can be rewritten as
\begin{equation}\label{eq:endpoints-general}
\text{(I)}+\text{(II)}
=
(p-1)(1+\ii^p)S_p
+p(\ii+\ii^{p-1})S_{p-1}S_1.
\end{equation}

\begin{lemma}[Endpoint dichotomy]
\label{lem:endpoint-dichotomy}
Let
\[
E_p:=
\sum_{j\in\{0,1,p-1,p\}}
\binom pj\ii^jS_{p-j}S_j.
\]
Then
\[
E_p=
\begin{cases}
(p-1)(1+\ii)
\left(S_p+\dfrac{p^2}{2}S_{p-1}\right),
&p\equiv1\pmod4,\\[3mm]
(p-1)(1-\ii)
\left(S_p-\dfrac{p^2}{2}S_{p-1}\right),
&p\equiv3\pmod4.
\end{cases}
\]
Moreover,
\[
E_p\equiv p^2(1+\ii)\pmod{p^3}
\qquad(p\equiv1\pmod4),
\]
whereas, for $p\ge7$ and $p\equiv3\pmod4$,
\[
E_p
\equiv
-\frac{p^5}{24}(p-1)^2(p-2)
B_{p-3}(1-\ii)
\pmod{p^6}.
\]
\end{lemma}

\begin{proof}
The exact formulas follow from
\eqref{eq:endpoints-general}, the identities
\[
\ii^p=
\begin{cases}
\ii,&p\equiv1\pmod4,\\
-\ii,&p\equiv3\pmod4,
\end{cases}
\qquad
\ii^{p-1}=
\begin{cases}
1,&p\equiv1\pmod4,\\
-1,&p\equiv3\pmod4,
\end{cases}
\]
and
\[
S_1=\frac{p(p-1)}2.
\]

If $p\equiv1\pmod4$, Lemma~\ref{lem:Sr-trunc} gives
\[
S_p+\frac{p^2}{2}S_{p-1}
\equiv
p^3B_{p-1}
\pmod{p^3}.
\]
Since
\[
pB_{p-1}\equiv-1\pmod p,
\]
we obtain
\[
E_p\equiv p^2(1+\ii)\pmod{p^3}.
\]

If $p\equiv3\pmod4$, then
\[
S_p-\frac{p^2}{2}S_{p-1}
\equiv
-\frac{p^5}{24}(p-1)(p-2)B_{p-3}
\pmod{p^6},
\]
because the $B_{p-1}$ terms cancel. Multiplication by
$(p-1)(1-\ii)$ gives the result.
\end{proof}

We now treat the two cases $p\equiv 1\pmod 4$ and $p\equiv 3\pmod 4$ separately.

\begin{lemma}[Bulk involution in the inert case]
\label{lem:bulk-involution}
Let $p\ge11$ satisfy $p\equiv3\pmod4$. For even
$4\le k\le p-5$, put
\[
A_k:=
\binom pk\ii^kS_{p-k}S_k
+
\binom p{k+1}\ii^{k+1}S_{p-k-1}S_{k+1}.
\]
Then
\[
A_k
\equiv
p^4C_k\ii^k(1+\ii)
B_kB_{p-1-k}
\pmod{p^6},
\]
where
\[
C_k=
\frac{p-k}{2k}\binom{p-1}{k-1}.
\]
Moreover, if
\[
k'=p-1-k,
\]
then
\[
C_{k'}=C_k,
\qquad
A_{k'}\equiv-A_k\pmod{p^6}.
\]
Consequently,
\[
\sum_{\substack{4\le k\le p-5\\k\text{ even}}}A_k
\equiv0\pmod{p^6}.
\]
\end{lemma}

\begin{proof}
For even $k$ in the stated range, Lemma~\ref{lem:Sr-trunc} gives
\[
S_k\equiv pB_k\pmod{p^3},
\qquad
S_{p-1-k}\equiv pB_{p-1-k}\pmod{p^3},
\]
and
\[
S_{p-k}
\equiv
\frac{p^2}{2}(p-k)B_{p-1-k}
\pmod{p^4},
\]
\[
S_{k+1}
\equiv
\frac{p^2}{2}(k+1)B_k
\pmod{p^4}.
\]
Since the two binomial coefficients are divisible by $p$, all omitted
terms contribute multiples of $p^6$. Hence
\[
A_k
\equiv
\frac{p^3}{2}B_kB_{p-1-k}
\left(
(p-k)\binom pk\ii^k
+
(k+1)\binom p{k+1}\ii^{k+1}
\right)
\pmod{p^6}.
\]
Using
\[
(k+1)\binom p{k+1}
=
(p-k)\binom pk
\]
and
\[
\binom pk
=
\frac pk\binom{p-1}{k-1},
\]
we obtain the stated formula.

Now let $k'=p-1-k$. Then
\[
\ii^{k'}
=
\ii^{p-1-k}
=
-\ii^k,
\]
because $p-1\equiv2\pmod4$ and $k$ is even. Also,
\[
\frac{p-k'}{k'}\binom{p-1}{k'-1}
=
\frac{p-k}{k}\binom{p-1}{k-1},
\]
so $C_{k'}=C_k$. Hence
\[
A_{k'}\equiv-A_k\pmod{p^6}.
\]
The asserted cancellation follows by summing over the involution
$k\mapsto p-1-k$.
\end{proof}

\subsection{Proof of Theorem for the case $p\equiv 1\pmod 4$}

\begin{theorem}
\label{thm:Gp-p1}
If $p\equiv 1\pmod 4$, then in $\ZZ[\ii]$,
\[
\G_p(p)\equiv p^2(1+\ii)\pmod{p^3}.
\]
Equivalently,
\[
v_p(\G_p(p))=2
\qquad\text{and}\qquad
\frac{\G_p(p)}{p^2}\equiv 1+\ii\pmod p.
\]
\end{theorem}

Below we prove this theorem. Since $p\equiv 1\pmod 4$ we have $\ii^p=\ii$ and $\ii^{p-1}=1$. For $p=5$, direct computation gives
\[
\G_5(5)=-7100(1+\ii)\equiv25(1+\ii)\pmod{125},
\]
so the theorem holds. We may therefore assume $p\ge13$.

By Lemma~\ref{lem:endpoint-dichotomy}, the endpoint and near-endpoint
terms contribute
\[
p^2(1+\ii)\pmod{p^3}.
\]
For every $2\le j\le p-2$, both $S_j$ and $S_{p-j}$ are divisible by
$p$, while $\binom pj$ is divisible by $p$. Hence every remaining term
is divisible by $p^3$. Therefore
\[
\G_p(p)\equiv p^2(1+\ii)\pmod{p^3}.
\]

\subsection{Proof of Theorem for the case $p\equiv 3\pmod 4$}
Here we will prove Theorem~\ref{thm:Gp-p3}. One can check by explicit computation that for $p=3$ Theorem~\ref{thm:Gp-p3} does not hold, $\G_3(3)=-27+27\ii$. 
So, here, \(p\ge 7\). Since $p\equiv 3\pmod 4$, we have $\ii^p=-\ii$ and $\ii^{p-1}=-1$.
In this case \eqref{eq:endpoints-general} becomes
\[
\text{(I)}+\text{(II)}=(p-1)(1-\ii)S_p+p(\ii-1)S_{p-1}S_1.
\]
Since $\ii-1=-(1-\ii)$ and $S_1=\frac{p(p-1)}{2}$, we may factor $(p-1)(1-\ii)$:
\begin{equation}\label{eq:endpoints-inert}
\text{(I)}+\text{(II)}=(p-1)(1-\ii)\Bigl(S_p-\frac{p^2}{2}S_{p-1}\Bigr).
\end{equation}

\paragraph{Step 1: the endpoint contribution.}
By Lemma~\ref{lem:endpoint-dichotomy},
\begin{equation}\label{eq:endpoints-result}
\textup{(I)}+\textup{(II)}
\equiv
-\frac{p^5}{24}(p-1)^2(p-2)B_{p-3}(1-\ii)
\pmod{p^6}.
\end{equation}
\paragraph{Step 2: the bulk contribution.}
After removing the boundary indices
\[
j\in\{2,3,p-3,p-2\},
\]
the remaining terms may be grouped as
\[
\textup{(III)}_{\mathrm{bulk}}
=
\sum_{\substack{4\le k\le p-5\\k\text{ even}}}A_k.
\]
For $p=7$ this sum is empty. For $p\ge11$,
Lemma~\ref{lem:bulk-involution} gives
\[
\textup{(III)}_{\mathrm{bulk}}
\equiv0\pmod{p^6}.
\]

\paragraph{Step 3: the boundary terms $j=2,3,p-3,p-2$.}
Write
\[
\text{(III)}_{\mathrm{bdry}}:=\sum_{j\in\{2,3,p-3,p-2\}}\binom{p}{j}\ii^{\,j}S_{p-j}S_j.
\]
Using $p\equiv 3\pmod 4$, we have
\[
\ii^2=-1,\qquad \ii^3=-\ii,\qquad \ii^{p-2}=\ii,\qquad \ii^{p-3}=1.
\]
Also $\binom{p}{p-2}=\binom{p}{2}$ and $\binom{p}{p-3}=\binom{p}{3}$, and $S_{p-j}S_j=S_jS_{p-j}$.
Thus
\begin{align*}
\text{(III)}_{\mathrm{bdry}}
&=
\binom{p}{2}(-1)S_{p-2}S_2+\binom{p}{p-2}\ii S_2S_{p-2}
+\binom{p}{3}(-\ii)S_{p-3}S_3+\binom{p}{p-3}(1)S_3S_{p-3}\\
&=(1-\ii)\Bigl(\binom{p}{3}S_{p-3}S_3-\binom{p}{2}S_{p-2}S_2\Bigr).
\end{align*}

Now we expand each factor to the precision needed modulo $p^6$.

First, $S_2$ and $S_3$ are exact:
\[
S_2=\sum_{t=1}^{p-1}t^2=\frac{p(p-1)(2p-1)}{6},
\qquad
S_3=\sum_{t=1}^{p-1}t^3=\Bigl(\frac{p(p-1)}{2}\Bigr)^2.
\]
Next, Lemma~\ref{lem:Sr-trunc} gives:
\[
S_{p-2}\equiv  \frac{p^2}{2!}(p-2)B_{p-3}    \pmod{p^4},
\qquad
S_{p-3}\equiv pB_{p-3}\pmod{p^3}.
\]

Compute the two products:
\[
\binom{p}{3}S_{p-3}S_3
\equiv
\frac{p(p-1)(p-2)}{6}\cdot pB_{p-3}\cdot \frac{p^2(p-1)^2}{4}
=
\frac{p^4}{24}(p-1)^3(p-2)B_{p-3}
\pmod{p^6},
\]
\[
\binom{p}{2}S_{p-2}S_2
\equiv
\frac{p(p-1)}{2}\cdot \frac{p^2}{2}(p-2)B_{p-3}\cdot \frac{p(p-1)(2p-1)}{6}
=
\frac{p^4}{24}(p-1)^2(p-2)(2p-1)B_{p-3}
\pmod{p^6}.
\]
Subtracting,
\begin{align*}
\binom{p}{3}S_{p-3}S_3-\binom{p}{2}S_{p-2}S_2
&\equiv
\frac{p^4}{24}(p-1)^2(p-2)\Bigl((p-1)-(2p-1)\Bigr)B_{p-3}\\
&=
-\frac{p^5}{24}(p-1)^2(p-2)B_{p-3}
\pmod{p^6}.
\end{align*}
Therefore
\begin{equation}\label{eq:bdry-result}
\text{(III)}_{\mathrm{bdry}}
\equiv
-\frac{p^5}{24}(p-1)^2(p-2)B_{p-3}\,(1-\ii)
\pmod{p^6}.
\end{equation}

\paragraph{Step 4: assemble all parts.}
We have
\[
\G_p(p)=(\text{(I)}+\text{(II)})+\text{(III)}_{\mathrm{bdry}}+\text{(III)}_{\mathrm{bulk}},
\]
and by \eqref{eq:endpoints-result}, \eqref{eq:bdry-result}, 
\[
\G_p(p)
\equiv
-\frac{p^5}{24}(p-1)^2(p-2)B_{p-3}(1-\ii)
-\frac{p^5}{24}(p-1)^2(p-2)B_{p-3}(1-\ii)
\pmod{p^6}.
\]
Thus
\begin{equation}\label{eq:final-inert}
\G_p(p)
\equiv
-\frac{p^5}{12}(p-1)^2(p-2)\,B_{p-3}\,(1-\ii)
\pmod{p^6}.
\end{equation}
The congruence implies
\[
p^5\mid\G_p(p)
\qquad\text{in }\ZZ[\ii].
\]
On the other hand, since $p\equiv3\pmod4$, Corollary~\ref{cor:location}
gives
\[
\G_p(p)=x(1-\ii)
\]
for some $x\in\ZZ$. Because the rational prime $p$ is coprime in
$\ZZ[\ii]$ to $1-\ii$, the divisibility $p^5\mid x(1-\ii)$ implies
$p^5\mid x$. Therefore
\[
(1-\ii)p^5\mid\G_p(p)
\qquad\text{in }\ZZ[\ii].
\]

This proves the congruence in Theorem~\ref{thm:Gp-p3}. The divisibility
and valuation assertions then follow as explained above.


\begin{remark}
By Theorem~\ref{thm:Gp-p3}, for $p\equiv 3\pmod 4$ one has
\[
v_p(\G_p(p))\ge 6
\quad\Longleftrightarrow\quad
B_{p-3}\equiv 0\pmod p,
\]
and hence $v_p(\G_p(p))=5$ is equivalent to $B_{p-3}\not\equiv 0\pmod p$.

Primes $p$ satisfying $B_{p-3}\equiv 0\pmod p$ are precisely the \emph{Wolstenholme primes}; equivalently,
\[
\binom{2p-1}{p-1}\equiv 1\pmod{p^4}.
\]
The two known examples are $p=16843$ and $p=2124679$; computational
searches have found no further Wolstenholme primes below $10^{11}$.
\end{remark}

\section{The cases \(p=3\) and \(p=5\)}

The exponential generating function of the Gaussian power sums is
\begin{equation}\label{eq:general-egf}
\mathcal E_p(t)
:=
\sum_{n\ge0}\G_n(p)\frac{t^n}{n!}
=
e^{\frac p2(1+\ii)t}
\frac{\sinh\!\bigl(\frac{p-1}{2}t\bigr)}{\sinh(t/2)}
\frac{\sin\!\bigl(\frac{p-1}{2}t\bigr)}{\sin(t/2)}.
\end{equation}
Indeed,
\[
\sum_{a=1}^{p-1}e^{at}
=
e^{pt/2}
\frac{\sinh\!\bigl(\frac{p-1}{2}t\bigr)}{\sinh(t/2)}
\]
and
\[
\sum_{b=1}^{p-1}e^{\ii bt}
=
e^{\ii pt/2}
\frac{\sin\!\bigl(\frac{p-1}{2}t\bigr)}{\sin(t/2)}.
\]
\begin{lemma}[Unique minimal coefficient]
\label{lem:unique-minimal-coefficient}
Let $q$ be an odd prime, and suppose
\[
E(t)=
C e^{\lambda t}
\sum_{m\ge0}u_m\frac{t^{4m}}{(4m)!}
=
\sum_{n\ge0}a_n\frac{t^n}{n!},
\]
where $C$ and every $u_m$ are units at every prime of $\ZZ[\ii]$
above $q$, and
\[
v_{\mathfrak q}(\lambda)=1
\]
for every such prime $\mathfrak q$. Assume that, for
$r\in\{0,1,2,3\}$ and $s\ge1$,
\[
v_q\!\left(\frac{(r+4s)!}{r!}\right)<4s.
\]
Then, for $n=4Q+r$,
\[
v_q(a_n)
=
r+v_q\binom nr.
\]
\end{lemma}

\begin{proof}
Coefficient extraction gives
\[
a_n
=
C\lambda^r
\sum_{s=0}^{Q}
\binom{n}{r+4s}
u_{Q-s}\lambda^{4s}.
\]
The term with $s=0$ has valuation
\[
r+v_q\binom nr.
\]
For $s\ge1$,
\[
v_q\binom{n}{r+4s}
\ge
v_q\binom nr
-
v_q\!\left(\frac{(r+4s)!}{r!}\right),
\]
so the $s$-th term has valuation strictly larger than that of the
$s=0$ term. Thus the latter is the unique term of minimal valuation.
\end{proof}

\begin{theorem}[The case \(p=3\)]\label{thm:p3}
For every \(n\ge1\),
\[
v_3\bigl(\G_n(3)\bigr)
=
\begin{cases}
0,
& n\equiv0\pmod4,\\[1mm]
1+v_3(n),
& n\equiv1\pmod4,\\[1mm]
2+v_3(n)+v_3(n-1),
& n\equiv2\pmod4,\\[1mm]
2+v_3(n)+v_3(n-1)+v_3(n-2),
& n\equiv3\pmod4.
\end{cases}
\]
Equivalently, if
\[
n=4Q+r,
\qquad
0\le r\le3,
\]
then
\[
v_3\bigl(\G_n(3)\bigr)
=
r+v_3\binom nr.
\]
\end{theorem}

\begin{proof}
Putting \(p=3\) in \eqref{eq:general-egf}, we obtain
\[
\mathcal E_3(t)
=
e^{\frac32(1+\ii)t}
\frac{\sinh t}{\sinh(t/2)}
\frac{\sin t}{\sin(t/2)}.
\]
Since
\[
\frac{\sinh t}{\sinh(t/2)}
=
2\cosh(t/2),
\qquad
\frac{\sin t}{\sin(t/2)}
=
2\cos(t/2),
\]
it follows that
\begin{equation}\label{eq:E3}
\mathcal E_3(t)
=
4e^{\alpha t}U_3(t),
\qquad
\alpha=\frac32(1+\ii),
\end{equation}
where
\[
U_3(t)
=
\cosh(t/2)\cos(t/2).
\]

We have
\[
\cosh(t/2)\cos(t/2)
=
\frac14
\sum_{\varepsilon,\delta\in\{\pm1\}}
e^{(\varepsilon+\delta\ii)t/2}.
\]
The four numbers
\[
\frac{\varepsilon+\delta\ii}{2},
\qquad
\varepsilon,\delta\in\{\pm1\},
\]
all have fourth power \(-1/4\). Therefore
\begin{equation}\label{eq:U3}
U_3(t)
=
\sum_{q\ge0}
\left(-\frac14\right)^q
\frac{t^{4q}}{(4q)!}.
\end{equation}

Here $4$ and all coefficients $(-1/4)^q$ are $3$-adic units, while
\[
v_3(\alpha)=1.
\]
For $r\in\{0,1,2,3\}$ and $s\ge1$, Legendre's formula gives
\[
v_3\!\left(\frac{(r+4s)!}{r!}\right)\le2s<4s.
\]
Lemma~\ref{lem:unique-minimal-coefficient} therefore yields
\[
v_3(\G_n(3))
=
r+v_3\binom nr
\qquad(n=4Q+r).
\]

For $r=0,1,2$, this immediately gives the first three displayed
formulas. For $r=3$,
\[
3+v_3\binom n3
=
2+v_3(n)+v_3(n-1)+v_3(n-2),
\]
because $v_3(3!)=1$.

\end{proof}

\begin{theorem}[The case \(p=5\)]\label{thm:p5}
For every \(n\ge1\),
\[
v_5\bigl(\G_n(5)\bigr)
=
\begin{cases}
0,
& n\equiv0\pmod4,\\[1mm]
1+v_5(n),
& n\equiv1\pmod4,\\[1mm]
2+v_5(n)+v_5(n-1),
& n\equiv2\pmod4,\\[1mm]
3+v_5(n)+v_5(n-1)+v_5(n-2),
& n\equiv3\pmod4.
\end{cases}
\]
Equivalently, if
\[
n=4Q+r,
\qquad
0\le r\le3,
\]
then
\[
v_5\bigl(\G_n(5)\bigr)
=
r+v_5\binom nr.
\]
\end{theorem}

\begin{proof}
Putting \(p=5\) in \eqref{eq:general-egf}, we obtain
\[
\mathcal E_5(t)
=
e^{\frac52(1+\ii)t}
\frac{\sinh(2t)}{\sinh(t/2)}
\frac{\sin(2t)}{\sin(t/2)}.
\]
Using
\[
\frac{\sinh(2t)}{\sinh(t/2)}
=
4\cosh(t/2)\cosh t
\]
and
\[
\frac{\sin(2t)}{\sin(t/2)}
=
4\cos(t/2)\cos t,
\]
we find
\begin{equation}\label{eq:E5}
\mathcal E_5(t)
=
16e^{\beta t}U_5(t),
\qquad
\beta=\frac52(1+\ii),
\end{equation}
where
\[
U_5(t)
=
\cosh(t/2)\cos(t/2)\cosh t\cos t.
\]

We first determine the coefficients of \(U_5(t)\). We have
\[
U_5(t)
=
\frac1{16}
\sum_{x\in\{\pm\frac12,\pm\frac32\}}e^{xt}
\sum_{y\in\{\pm\frac12,\pm\frac32\}}e^{\ii yt}.
\]
Thus
\[
U_5(t)
=
\frac1{16}\sum_{z\in S}e^{zt},
\]
where
\[
S=
\left\{
x+\ii y:
x,y\in
\left\{\pm\frac12,\pm\frac32\right\}
\right\}.
\]
The set \(S\) is invariant under multiplication by \(\ii\), so only
powers \(t^{4q}\) occur. Representatives for its four orbits under
multiplication by \(\ii\) are
\[
z_1=\frac{1+\ii}{2},
\qquad
z_2=\frac{3+3\ii}{2},
\qquad
z_3=\frac{3+\ii}{2},
\qquad
z_4=\frac{1+3\ii}{2}.
\]
It follows that
\[
U_5(t)
=
\sum_{q\ge0}c_q\frac{t^{4q}}{(4q)!},
\]
where
\[
c_q
=
\frac14
\left(
z_1^{4q}+z_2^{4q}+z_3^{4q}+z_4^{4q}
\right).
\]
A direct computation gives
\[
z_1^4=-\frac14,
\qquad
z_2^4=-\frac{81}{4},
\qquad
z_3^4=\frac{7+24\ii}{4},
\qquad
z_4^4=\frac{7-24\ii}{4}.
\]
Hence
\begin{equation}\label{eq:cq}
c_q
=
\frac{
(-1)^q+(-81)^q+(7+24\ii)^q+(7-24\ii)^q
}{4^{q+1}}.
\end{equation}

We now work with the two Gaussian primes
\[
\pi=2+\ii,
\qquad
\bar\pi=2-\ii,
\qquad
5=\pi\bar\pi.
\]
Let \(v_\pi\) and \(v_{\bar\pi}\) be the corresponding normalized
discrete valuations on \(\QQ(\ii)\). For every nonzero
\(z\in\ZZ[\ii]\),
\begin{equation}\label{eq:v5-min}
v_5(z)
=
\min\{v_\pi(z),v_{\bar\pi}(z)\}.
\end{equation}
Moreover, for every rational integer \(m\ne0\),
\[
v_\pi(m)=v_{\bar\pi}(m)=v_5(m).
\]

We claim that every \(c_q\) is a unit at both \(\pi\) and \(\bar\pi\).
For \(q=0\), this follows from \(c_0=1\). Suppose \(q\ge1\). By the
Chinese remainder theorem,
\[
\ZZ[\ii]/(5)
\cong
\ZZ[\ii]/(\pi)\times\ZZ[\ii]/(\bar\pi)
\cong
\FF_5\times\FF_5,
\]
where the components corresponding to $(\pi)$ and $(\bar\pi)$ are
obtained respectively by substituting
\[
\ii=-2
\qquad\text{and}\qquad
\ii=2.
\] Under this isomorphism,
\[
7+24\ii\longmapsto(4,0),
\qquad
7-24\ii\longmapsto(0,4),
\]
whereas
\[
-1\longmapsto(4,4),
\qquad
-81\longmapsto(4,4).
\]
Therefore the numerator in \eqref{eq:cq} maps to
\[
\bigl(3\cdot4^q,\,3\cdot4^q\bigr)
=
\bigl(3(-1)^q,\,3(-1)^q\bigr),
\]
which is nonzero in both components. Since \(4^{q+1}\) is also a unit
at both primes, we conclude that
\begin{equation}\label{eq:cq-units}
v_\pi(c_q)=v_{\bar\pi}(c_q)=0
\qquad
(q\ge0).
\end{equation}

Also,
\[
v_\pi(\beta)=v_{\bar\pi}(\beta)=1,
\]
because \(1+\ii\) is divisible by neither \(\pi\) nor \(\bar\pi\).

The constant $16$ and every $c_q$ are units at both primes above $5$,
and
\[
v_\pi(\beta)=v_{\bar\pi}(\beta)=1.
\]
Moreover, for $r\in\{0,1,2,3\}$ and $s\ge1$, Legendre's formula gives
\[
v_5\!\left(\frac{(r+4s)!}{r!}\right)\le s<4s.
\]
Lemma~\ref{lem:unique-minimal-coefficient}, applied at both primes
above $5$, therefore yields
\[
v_5(\G_n(5))
=
r+v_5\binom nr
\qquad(n=4Q+r).
\]
Since $r!$ is a $5$-adic unit for $0\le r\le3$,
\[
v_5\binom nr
=
\sum_{j=0}^{r-1}v_5(n-j),
\]
which gives the four displayed formulas.
\end{proof}

\section{Data Availability Statement} 
The numerical values reported in this article were obtained by direct
exact computation in $\ZZ[\ii]$. No external datasets were used.
Code used for the computations is available from the authors upon
reasonable request.


\end{document}